\documentclass[11pt]{article}
\usepackage{mathbbold}
\usepackage{amsfonts}
\usepackage{mathrsfs}
\usepackage{bbm}
\usepackage{latexsym,amsfonts,amssymb,amsmath,amsthm}
\usepackage{graphicx}

\usepackage{setspace}

\newtheorem{theorem}{Theorem}[section]

\newtheorem{lemma}[theorem]{Lemma}
\newtheorem{corollary}[theorem]{Corollary}
\newtheorem{remark}[theorem]{Remark}

\numberwithin{equation}{section}
\def\exp{\mathop{\hbox{exp}}}

\begin{document}

\baselineskip 20pt

\begin{center}

\textbf{\Large Synchronization of Coupled Stochastic Systems Driven
by Non-Gaussian L\'evy Noises} \footnote{This work has been
partially supported by NSFC Grants 11071165 and 11071199, NSF of
Guangxi Grants 2013GXNSFBA019008 and Guangxi Provincial Department
of Research Project Grants 2013YB102.\\
$^*$Corresponding author: A. Gu~(mathgah@gmail.com).}

\vskip 0.5cm

{\large  Anhui Gu, Yangrong Li}

\vskip 0.3cm

\textit{School of Mathematics and Statistics, Southwest
University, Chongqing, 400715, China}\\

\vskip 1cm

\begin{minipage}[c]{15cm}

\noindent \textbf{Abstract}: We consider the synchronization of the
solutions to coupled stochastic systems of $N$-stochastic ordinary
differential equations (SODEs) driven by Non-Gaussian L\'evy noises
($N\in \mathbb{N})$. We discuss the synchronization between two
solutions and among different components of solutions under certain
dissipative and integrability conditions. Our results generalize the
present work obtained in Liu et al (2010) and Shen et al (2010).
\vspace{5pt}

\noindent\textit{MSC}: 60H10, 34F05, 37H10

\vspace{5pt}

\textit{Keywords}:  Synchronization; L\'evy noise; Skorohod metric;
random attractor; c\`adl\`ag random dynamical system.

\end{minipage}
\end{center}

\vspace{10pt}
%\newpage

\baselineskip 18pt

\section{Introduction}

The synchronization of coupled systems is a well-known phenomenon in
both biology and physics. Description of its diversity of occurrence
can be founded in \cite{ACH}, \cite{AL}, \cite{AVR}, \cite{AR},
\cite{Strogatz}, \cite{PRK}, \cite{Glass}. Synchronization of
deterministic coupled systems has been investigated mathematically
in \cite{AR}, \cite{CRD}, \cite{PZ} for autonomous cases and in
\cite{Kloeden} for non-autonomous systems. For the stochastic cases,
we can refer to the coupled system of It\^{o} SODEs with additive
noise \cite{CK}, \cite{CCK} and multiplicative noise \cite{CKN},
\cite{SZH}. Recently, Shen et al. \cite{SZH} generalized the
multiplicative case to $N$-Stratonovich SODEs. These dissipative
dynamical systems discussed above are focused on the Gaussian noises
(in terms of Brownian motion). However, complex systems in
engineering and science are often subjected to non-Gaussian
fluctuations or uncertainties. The coupled dynamical systems under
non-Gaussian L\'evy noises are considered in \cite{LDLK1},
\cite{LDLK2} and \cite{Gu}.

Let $(\Omega, \mathcal{F}, \mathbb{P})$ be a probability space,
where $\Omega=D(\mathbb{R}, \mathbb{R}^d)$ of c\`adl\`ag functions
with the Skorohod metric as the canonical sample space and denote by
$\mathcal{F}:=\mathcal{B}(D(\mathbb{R}, \mathbb{R}^d))$ the Borel
$\sigma$-algebra on $\Omega$. Let $\mu_L$ be the (L\'evy)
probability measure on $\mathcal{F}$ which is given by the
distribution of a two-sided L\'evy process with paths in $\Omega$,
i.e. $\omega(t)=L_t(\omega)$.

Define $\theta=(\theta_t, t\in \mathbb{R})$ on $\Omega$ the shift by
$$(\theta_t\omega)(s):=\omega(t+s)-\omega(t).$$
Then the mapping $(t, \omega)\rightarrow \theta_t\omega$ is
continuous and measurable \cite{Arnold}, and the (L\'evy)
probability measure is $\theta$-invariant, i.e.
$$\mu_L(\theta_t^{-1}(A))=\mu_L(A),$$
for all $A\in \mathcal{F}$, see \cite{Applebaum} for more details.
Consider the following SODEs system driven by non-Gaussian L\'evy
noises in $\mathbb{R}^{Nd}$,
\begin{equation}
dX_t^{(j)}=f^{(j)}(X_t^{(j)})dt+c_jdL_t^{(j)}, \ \ j=1, \cdots,
N,\label{1.1}\end{equation} where $c_j\in \mathbb{R}^d$, are
constants vectors with no components equal to zero, $L_t^{(j)}$ are
independent two-sided scalar L\'evy processes on $(\Omega,
\mathcal{F}, \mathbb{P})$ satisfying proper conditions which will be
specified later, and $f^{(j)}, j=1, \cdots, N,$ are regular enough
to ensure the existence and uniqueness of solutions and satisfy the
one-sided dissipative Lipschitz conditions
\begin{equation}
\langle x_1-x_2, f^{(j)}(x_1)-f^{(j)}(x_2)\rangle\leq
-l\|x_1-x_2\|^2, \ \ j=1, \cdots, N \label{1.2}\end{equation} on
$\mathbb{R}^d$ for some $l>4$. In addition to \eqref{1.2}, we
further assume the following integrability condition: There exists
$m_0>0$ such that for any $m\in (0, m_0]$, and any c\`adl\`ag
function $X: \mathbb{R}\rightarrow\mathbb{R}^d$ with sub-exponential
growth it follows
\begin{equation}
\int^t_{-\infty}e^{ms}|f^{(j)}(X(s))|^2ds<\infty, \ \ j=1, \cdots,
N. \label{1.3}\end{equation} Without lose of generality, we also
assume the Lipschitz constant $l\leq m_0$.

Set $$x^{(j)}(t, \omega)=X_t^{(j)}-\bar{X}_t^{(j)}, \ \ t\in
\mathbb{R}, \omega\in \Omega, j=1, \cdots, N,$$ where
$$\bar{X}_t^{(j)}=c_je^{-t}\int^t_{-\infty}e^{s}dL_s^{(j)}, \ \ j=1, \cdots, N,$$
are the stationary solutions of the Langevin equations
$$dX_t^{(j)}=-X_t^{(j)}dt+c_jdL_t^{(j)}, \ \ j=1, \cdots, N.$$
Then system \eqref{1.1} can be translated into the following random
ordinary differential equations (RODEs), with right-hand derivative
in time
\begin{eqnarray}
\frac{dx^{(j)}}{dt_+}&=&F^{(j)}(x^{(j)}, \bar{X}_t^{(j)})\nonumber\\
&:=&f^{(j)}(x^{(j)}+\bar{X}_t^{(j)})+x^{(j)}+\bar{X}_t^{(j)}, \ j=1,
\cdots, N.\label{1.4}
\end{eqnarray}
Now we consider the linear coupled RODEs of \eqref{1.4}
\begin{equation}
\frac{dx^{(j)}}{dt_+}=F^{(j)}(x^{(j)}, \bar{X}_t^{(j)})
+\lambda(x^{(j-1)}-2x^{(j)}+x^{(j+1)}), \ j=1, \cdots, N,
\label{1.5}\end{equation} with the coupled coefficient $\lambda>0$,
where $x^{(0)}=x^{(N)}$ and $x^{(N+1)}=x^{(1)}$. Hence \eqref{1.5}
can be written as the following equivalent SODEs
\begin{eqnarray}
dX_t^{(j)}&=&f^{(j)}(X_t^{(j)})+\lambda(X_t^{(j-1)}-2X_t^{(j)}
+X_t^{(j+1)})-\lambda(\bar{X}_t^{(j-1)}-2\bar{X}_t^{(j)}+\bar{X}_t^{(j+1)})
\nonumber\\
&&+c_j dL_t^{(j)}, \ j=1, \cdots, N,\label{1.6}
\end{eqnarray}
where $X_t^{(0)}=X_t^{(N)}$ and $X_t^{(N+1)}=X_t^{(1)}$. For
synchronization of solutions to RODEs system \eqref{1.5}, there are
two cases: one for any two solutions and the other for components of
solutions. When $N=2$, Liu et al. \cite{LDLK1} consider both types
of synchronization. Under the one-sided dissipative Lipschitz
condition \eqref{1.2} and the integrability condition \eqref{1.3},
they firstly proved that synchronization of any two solutions occurs
and the random dynamical system generated by the solution of
\eqref{1.5}$_{N=2}$ has a singleton sets random attractor, then they
obtained that the synchronization between any two components of
solutions occurs as the coupled coefficient $\lambda$ tends to
infinity. The synchronization result implies that coupled dynamical
system share a dynamical feature in some asymptotic sense. Based on
the work of \cite{LDLK1} and \cite{SZH}, we consider the
synchronization of solutions of \eqref{1.5} in the case of $N\geq 3$
and obtain the similar results. We show that the random dynamical
system (RDS) generated by the solution of the coupled RODEs system
\eqref{1.5} has a singleton sets random attractor which implies the
synchronization of any two solutions of \eqref{1.5}. Moreover, the
singleton set random attractor determines a stationary stochastic
solution of the equivalently coupled SODEs system \eqref{1.6}. We
also show that any two solutions of RODEs system \eqref{1.5}
converge to a solution $Z(t, \omega)$ of the averaged RODE
\begin{equation}
\frac{dZ}{dt_+}=\frac{1}{N}\sum_{j=1}^Nf^{(j)}(\bar{X}_t^{(j)}+Z)
+\frac{1}{N}\sum_{j=1}^N(\bar{X}_t^{(j)}+Z),\label{1.7}\end{equation}
as the coupling coefficient $\lambda\rightarrow\infty$. It is worth
mentioning that the generalization is not trivial because new
techniques similar to \cite{SZH} are needed.

\section{Auxiliary Lemmas}
We will frequently use the following auxiliary results.
\begin{lemma} \label{property}
 {\cite{LDLK1}}\ (Pathwise boundedness and convergence.) \
Let $L_t$ be a two-sided L\'evy motion on $\mathbb{R}^d$ for which
$\mathbb{E}|L_1|<\infty$ and $\mathbb{E}|L_1|=\gamma$.
Then we have\\
(A) $\lim_{t\rightarrow\pm \infty}\frac{1}{t}L_t=\gamma$, a.s.\\
(B) the integrals $\int_{-\infty}^t e^{-\delta(t-s)}dL_s(\omega)$
are pathwisely uniformly bounded in $\delta>0$ on finite time
intervals $[T_1, T_2]$ in $\mathbb{R}$;\\
(C) the integrals $\int_{T_1}^t e^{-\delta(t-s)}dL_s(\omega)
\rightarrow 0$ as $\delta\rightarrow\infty$, pathwise on finite time
intervals $[T_1, T_2]$ in $\mathbb{R}$.
\end{lemma}

\begin{lemma} \label{Gronwall} (Gronwall type inequality.)
Suppose that $D(t)$ is a $n\times n$ matrix and $\Phi(t), \Psi(t)$
are $n$-dimensional vectors on $[T_0, T]\ (T\geq T_0,\ T, T_0\in
\mathbb{R})$ which are sufficiently regular. If the following
inequality holds in the componentwise sense
\begin{equation}
\frac{d}{dt_+}\Phi(t)\leq D(t)\Phi(t)+\Psi(t), \ t\geq T_0,
\label{2.1}\end{equation} where
$\frac{d}{dt_+}\Phi(t):=\lim_{h\downarrow
0^+}\frac{\Phi(t+h)-\Phi(t)}{h}$ is right-hand derivative of
$\Phi(t)$. Then
\begin{equation}
\Phi(t)\leq \exp(\int_{T_0}^tD(s)ds)\Phi(T_0)+\int_{T_0}^t
\exp(\int_{\tau}^tD(s)ds)\Psi(\tau)d\tau, \ t\geq T_0.
\label{2.2}\end{equation}
\end{lemma}

\begin{proof}
See Lemma 2.8 in \cite{Robinson} and the proof of Lemma 2.2 in
\cite{SZH}.
\end{proof}

\begin{lemma} \label{attractor} {\cite{LDLK1}}\
(Random attractor for c\`adl\`ag RDS.) Let $(\theta, \phi)$ be an
RDS on $\Omega\times \mathbb{R}^d$ and let $\phi$ be continuous in
space, but c\`adl\`ag in time. If there exists a family
$B=\{B(\omega), \omega\in \Omega\}$ of non-empty measurable compact
subsets $B(\omega)$ of $\mathbb{R}^d$ and a $T_{D, \omega}\geq 0$
such that
$$\phi(t, \theta_{-t}\omega, D(\theta_{-t}\omega))\subset B(\omega),
\ \forall t\geq T_{D, \omega},$$ for all families $D=\{D(\omega),
\omega\in \Omega\}$ in a given attracting universe, then the RDS
$(\theta, \phi)$ has a random attractor
$\mathcal{A}=\{\mathcal{A}(\omega), \omega\in \Omega\}$ with the
component subsets defined for each $\omega\in \Omega$ by
$$\mathcal{A}(\omega)=\bigcap_{s>0}\overline{\bigcup_{t\geq s}
\phi(t, \theta_{-t}\omega, B(\theta_{-t}\omega))}.$$ Furthermore, if
the random attractor consist of singleton sets, i.e.
$\mathcal{A}(\omega)=\{X^*(\omega)\}$ for some random variable
$X^*$, then $X^*_t(\omega)=X^*_t(\theta_t\omega)$ is a stationary
stochastic process.
\end{lemma}

\section{Synchronization of Two Solutions}
Consider the coupled RODEs system \eqref{1.5}
\begin{equation}
\frac{dx^{(j)}}{dt_+}=F^{(j)}(x^{(j)}, \bar{X}_t^{(j)})
+\lambda(x^{(j-1)}-2x^{(j)}+x^{(j+1)}), \ j=1, \cdots, N,
\label{3.1}\end{equation} with initial data
\begin{equation}
x^{(j)}(0, \omega)=x^{(j)}_0(\omega)\in \mathbb{R}^d, \ \omega\in
\Omega, \ j=1, \cdots, N, \label{3.2}\end{equation} where
$\lambda>0$, and
\begin{equation}
F^{(j)}(x^{(j)}, \bar{X}_t^{(j)}):=f^{(j)}(x^{(j)}
+\bar{X}_t^{(j)})+x^{(j)}+\bar{X}_t^{(j)}, \ j=1, \cdots, N.
\label{3.3}\end{equation} Here $f^{(j)}$ are regular enough to
ensure the existence and uniqueness of global solutions on
$\mathbb{R}$ and satisfy the one-sided dissipative Lipschitz
condition \eqref{1.2} and integrability condition \eqref{1.3} for
$j=1, \cdots, N$.

First, we have the result of existence of stationary solutions.

\begin{lemma} \label{existence}
Supposed the assumptions \eqref{1.2} and \eqref{1.3} be satisfied.
Then the coupled RODEs system \eqref{3.1} with initial condition
\eqref{3.2} has a unique stationary solution.
\end{lemma}

\begin{proof}
For any two solutions $(x_1^{(1)}(t), x_1^{(2)}(t), \cdots,
x_1^{(N)}(t))^\mathbf{T}$ and $(x_2^{(1)}(t), x_2^{(2)}(t),
\\\cdots, x_2^{(N)}(t))^\mathbf{T}$ of RODEs system
\eqref{3.1}-\eqref{3.2}. By the dissipative Lipschitz condition
\eqref{1.2}, for $j=1, \cdots, N$, we have
\begin{eqnarray}
\frac{d}{dt_+}\|x_1^{(j)}(t)-x_2^{(j)}(t)\|^2&=&2\langle
x_1^{(j)}(t)-x_2^{(j)}(t), \frac{d}{dt_+}x_1^{(j)}(t)
-\frac{d}{dt_+}x_2^{(j)}(t) \rangle\nonumber\\
&=&2\langle f^{(j)}(x_1^{(j)}+\bar{X}_t^{(j)})-f^{(j)}(x_2^{(j)}
+\bar{X}_t^{(j)}), x_1^{(j)}(t)-x_2^{(j)}(t)\rangle\nonumber\\
&&+(2-4\lambda)\|x_1^{(j)}(t)-x_2^{(j)}(t)\|^2\nonumber\\
&&+2\lambda \langle x_1^{(j-1)}(t)-x_2^{(j-1)}(t), x_1^{(j)}(t)
-x_2^{(j)}(t)\rangle\nonumber\\
&&+2\lambda \langle x_1^{(j+1)}(t)-x_2^{(j+1)}(t), x_1^{(j)}(t)
-x_2^{(j)}(t)\rangle\nonumber\\
&\leq& (2-2l-2\lambda)\|x_1^{(j)}(t)-x_2^{(j)}(t)\|^2\nonumber\\
&&+\lambda \|x_1^{(j-1)}(t)-x_2^{(j-1)}(t)\|^2\nonumber\\
&&+\lambda \|x_1^{(j+1)}(t)-x_2^{(j+1)}(t)\|^2.\label{3.4}
\end{eqnarray}
Define for $t\in \mathbb{R}$,
$$\mathbf{x}(t)=(\|x_1^{(1)}(t)-x_2^{(1)}(t)\|^2, \|x_1^{(2)}(t)
-x_2^{(2)}(t)\|^2, \cdots, \|x_1^{(N)}(t)
-x_2^{(N)}(t)\|^2)^{\mathbf{T}},$$ and
$$D_{\lambda}= \left ( \begin{array}{cccccc}
     2-2l-2\lambda &\lambda  &0  &\cdots &0 &\lambda \\
    \lambda & 2-2l-2\lambda & \lambda &0 &\cdots & 0\\
    0&\lambda&2-2l-2\lambda &\ddots &\ddots & \vdots\\
    \vdots& \ddots& \ddots & \ddots &\lambda &0\\
    0& \cdots & 0 & \lambda &2-2l-2\lambda & \lambda\\
    \lambda & 0& \cdots & 0 & \lambda & 2-2l-2\lambda
\end{array} \right)_{N\times N}.$$
Thus, the differential inequalities can be written as a simple form
\begin{equation}
\mathbf{\dot{x}}(t)\leq D_{\lambda}\mathbf{x}(t),\
\mbox{-componentwise}.\label{3.5}\end{equation} By Lemma
\ref{Gronwall}, it yields from \eqref{3.5} that
\begin{equation}
\mathbf{x}(t)\leq\exp(\int_0^tD_{\lambda}ds) \mathbf{x}(0), \
\mbox{-componentwise}.\label{3.6}\end{equation}

Now, we firstly to estimate the upper bound of eigenvalues of the
real symmetric matrix $\int_0^tD_{\lambda}ds$. The quadratic from
satisfies
\begin{eqnarray*}
f(\zeta_1, \zeta_2, \cdots, \zeta_N)&=&\zeta^{\mathbf{T}}
(\int_0^tD_{\lambda}ds)\zeta\\
&=&(2-2l-2\lambda)t\sum_{j=1}^N\zeta_j^2+2\lambda t
\sum_{j=1}^N\zeta_j\zeta_{j-1}\\
&\leq & (2-l)t\sum_{j=1}^N\zeta_j^2-lt\sum_{j=1}^N\zeta_j^2,
\end{eqnarray*}
where $\zeta=(\zeta_1, \zeta_2, \cdots, \zeta_N)^{\mathbf{T}}\in
\mathbb{R}^N$ and $\zeta_0=\zeta_N$. Due to the Lipschitz constant
$l>4$, we have
$$f(\zeta_1, \zeta_2, \cdots, \zeta_N)\leq -lt\sum_{j=1}^N\zeta_j^2,$$
which implies that the quadratic form is negative definite and
eigenvalues of $\int_0^tD_{\lambda}ds$ satisfy
\begin{equation}
\max\{\mu^{(1)}_{\lambda}, \mu^{(2)}_{\lambda}, \cdots,
\mu^{(N)}_{\lambda}\}\leq -lt.\label{3.7}\end{equation}

Because of the real and symmetric properties of matrix
$\int_0^tD_{\lambda}ds$, for $j=1, \cdots, N$, we obtain
\begin{eqnarray}
\|\exp(\int_0^tD_{\lambda}ds)\mathbf{x}(0)\|^2&\leq&
\|\mathbf{x}(0)\|^2 \exp(2\max\{\mu^{(1)}_{\lambda},
\mu^{(2)}_{\lambda}, \cdots,
\mu^{(N)}_{\lambda}\})\nonumber\\
&\leq &\|\mathbf{x}(0)\|^2\exp(-2lt),\label{3.8}
\end{eqnarray}
which leads to
$$\lim_{t\rightarrow\infty}\|x_1^{(j)}(t)-x_2^{(j)}(t)\|=0,
\ j=1, \cdots, N,$$ that is, all solutions of the coupled RODEs
system \eqref{3.1}-\eqref{3.2} converge pathwise to each other as
time $t$ tends to infinity. The proof is finished.
\end{proof}

Now, we use the theory of random dynamical systems which generated
by SDEs driven by L\'evy motion to find what the solutions of
\eqref{3.1}-\eqref{3.2} will converge to. It is easy to see from
\cite{LDLK1} that the solution
$$\phi(t, \omega)=(x^{(1)}(t, \omega), x^{(2)}(t, \omega), \cdots,
x^{(N)}(t, \omega))^{\mathbf{T}}, \ \omega\in \Omega$$ of system
\eqref{3.1}-\eqref{3.2} generates a c\`adl\`ag RDS over $(\Omega,
\mathcal{F}, \mathbb{P}, (\theta_t)_{t\in \mathbb{R}})$ with state
space $\Omega\times\mathbb{R}^{Nd}$. The RDS $(\theta, \phi)$ is
continuous in space but c\`adl\`ag in time. Recall that a stationary
solution $X^*$ is a stationary solution of a stochastic differential
equation system may be characterized as a stationary orbit of the
corresponding RDS $(\theta, \phi)$ generated by the stochastic
differential equation system, namely, $\phi(t,
\omega)X^*(\omega)=X^*(\theta_t\omega)$.

Then, we have the result for this RDS.

\begin{theorem} \label{sys1}
Under the conditions \eqref{1.2} and \eqref{1.3}, the RDS $\phi(t,
\omega), t\in \mathbb{R}, \omega\in \Omega$, has a singleton sets
random attractor given by
$$\mathcal{A}_{\lambda}(\omega)=\{(\bar{x}_{\lambda}^{(1)}(\omega),
\bar{x}_{\lambda}^{(2)}(\omega), \cdots,
\bar{x}_{\lambda}^{(N)}(\omega))^{\mathbf{T}}\},$$ which implies the
synchronization of any two solutions of system
\eqref{3.1}-\eqref{3.2}. Furthermore,
$$(\bar{x}_{\lambda}^{(1)}(\theta_t\omega)+\bar{X}_t^{(1)},
\bar{x}_{\lambda}^{(2)}(\theta_t\omega)+\bar{X}_t^{(2)}, \cdots,
\bar{x}_{\lambda}^{(N)}(\theta_t\omega)+\bar{X}_t^{(N)})^{\mathbf{T}}$$
is the stationary stochastic solution of the equivalent coupled
SODEs \eqref{1.6}.
\end{theorem}

\begin{proof} For $j=1, \cdots, N,$ we have
\begin{eqnarray*}
\frac{d}{dt_+}\|x^{(j)}(t)\|^2&=&2\langle x^{(j)}(t),
\frac{d}{dt_+}x^{(j)}(t)\rangle\\
&=& 2\langle f^{(j)}(x^{(j)}(t)+\bar{X}_t^{(j)}), x^{(j)}(t)
\rangle+2\langle x^{(j)}(t)+\bar{X}_t^{(j)}, x^{(j)}(t)\rangle\\
&&-4\lambda\|x^{(j)}(t)\|^2+2\lambda \langle x^{(j)}(t),
x^{(j-1)}(t)\rangle+2\lambda \langle x^{(j)}(t),
x^{(j+1)}(t)\rangle\\
&\leq&  2\langle f^{(j)}(x^{(j)}(t)+\bar{X}_t^{(j)})
-f^{(j)}(\bar{X}_t^{(j)}), x^{(j)}(t)\rangle+2\langle f^{(j)}
(\bar{X}_t^{(j)}),  x^{(j)}(t)\rangle\\
&&+(2-4\lambda)\|x^{(j)}(t)\|^2+2\langle \bar{X}_t^{(j)},
x^{(j)}(t)\rangle\\
&&+2\lambda \langle x^{(j)}(t), x^{(j-1)}(t)\rangle
+2\lambda \langle x^{(j)}(t), x^{(j+1)}(t)\rangle\\
&\leq& \|\bar{X}_t^{(j)}\|^2+|f^{(j)}(\bar{X}_t^{(j)})|^2
+(4-2l-2\lambda)\|x^{(j)}(t)\|^2\\
&&+\lambda\|x^{(j-1)}(t)\|^2+\lambda\|x^{(j+1)}(t)\|^2.
\end{eqnarray*}
Analogous to \eqref{3.5}, we get
$$\mathbf{\dot{y}}(t)\leq \tilde{D}_{\lambda}\mathbf{y}(t)
+\mathbf{g}(t),$$ where
$$\mathbf{y}(t)=(\|x^{(1)}(t)\|^2, \|x^{(2)}(t)\|^2, \cdots,
\|x^{(N)}(t)\|^2)^{\mathbf{T}}, \ t\in \mathbb{R},$$
\begin{eqnarray*}\mathbf{g}(t)&=&(|f^{(1)}(\bar{X}_t^{(1)})|^2
+\|\bar{X}_t^{(1)}\|^2, |f^{(2)}(\bar{X}_t^{(2)})|^2\\
&&\ \ \ \ \ \ \ \ \ \ \ +\|\bar{X}_t^{(2)}\|^2, \cdots,
|f^{(N)}(\bar{X}_t^{(N)})|^2+\|\bar{X}_t^{(N)}\|^2,)^{\mathbf{T}}, \
t\in \mathbb{R},\end{eqnarray*} and
$$\tilde{D}_{\lambda}= \left ( \begin{array}{cccccc}
     4-2l-2\lambda &\lambda  &0  &\cdots &0 &\lambda \\
    \lambda & 4-2l-2\lambda & \lambda &0 &\cdots & 0\\
    0&\lambda&4-2l-2\lambda &\ddots &\ddots & \vdots\\
    \vdots& \ddots& \ddots & \ddots &\lambda &0\\
    0& \cdots & 0 & \lambda &4-2l-2\lambda & \lambda\\
    \lambda & 0& \cdots & 0 & \lambda & 4-2l-2\lambda
\end{array} \right)_{N\times N}.$$
Then by Lemma \ref{Gronwall},
$$\mathbf{y}(t)\leq \exp(\int_{t_0}^t
\tilde{D}_{\lambda}ds)\mathbf{y}(t_0)+\int_{t_0}^t\exp(\int_{\tau}^t
\tilde{D}_{\lambda}ds)\mathbf{g}(\tau)d\tau, \ t\geq t_0.$$ Similar
to Lemma \ref{existence}, we have
$$\|\exp(\int_{t_0}^t\tilde{D}_{\lambda}ds)\mathbf{y}(t_0)\|
\leq \|\mathbf{y}(t_0)\|\exp(-l(t-t_0)),\ t\geq t_0.$$ Define
\begin{equation} \rho_{\lambda}(\omega):=\int_{-\infty}^0
\exp(\int_{\tau}^0\tilde{D}_{\lambda}ds)\mathbf{g}(\tau)d\tau,
\label{3.9}\end{equation} and
\begin{equation}
R_{\lambda}^2(\omega)=1+\|\rho_{\lambda}(\omega)\|^2,
\label{3.10}\end{equation} and let $\mathbb{B}_{\lambda}$ be a
random ball in $\mathbb{R}^{Nd}$ centered at the origin with radius
$R_{\lambda}(\omega)$. Obviously, the infinite integral on the
right-hand side of \eqref{3.9} is well-defined by Lemma
\ref{property} and the integrability condition \eqref{1.3}. Hence by
Lemma 2.3, the coupled system has a random attractor
$\mathcal{A}_{\lambda}=\{\mathcal{A}_{\lambda}(\omega), \omega\in
\Omega\}$ with $\mathcal{A}_{\lambda}(\omega)\subset
\mathbb{B}_{\lambda}$. By Lemma \ref{existence}, all solutions of
\eqref{3.1}-\eqref{3.2} converge pathwise to each other, therefore,
$\mathcal{A}_{\lambda}(\omega)$ consists of singleton sets, that is
$$\mathcal{A}_{\lambda}(\omega)=\{(\bar{x}_{\lambda}^{(1)}(\omega),
\bar{x}_{\lambda}^{(2)}(\omega), \cdots, \bar{x}_{\lambda}^{(N)}
(\omega))^{\mathbf{T}}\}.$$

We transform the coupled RODEs \eqref{3.1} back to the coupled SODEs
\eqref{1.6}, the corresponding pathwise singleton sets attractor is
then equal to
$$(\bar{x}_{\lambda}^{(1)}(\theta_t\omega)+\bar{X}_t^{(1)},
\bar{x}_{\lambda}^{(2)}(\theta_t\omega)+\bar{X}_t^{(2)}, \cdots,
\bar{x}_{\lambda}^{(N)}(\theta_t\omega)+\bar{X}_t^{(N)})^{\mathbf{T}},$$
which is exactly a stationary stochastic solution of the coupled
SODEs \eqref{1.6} because the Ornstein-Uhlenbeck process is
stationary.
\end{proof}

\section{Synchronization of Components of Solutions}
It is known in Section 3 that all solutions of the coupled RODEs
system \eqref{3.1}-\eqref{3.2} converge pathwise to each other in
the future for a fixed positive coupling coefficient $\lambda$.
Here, we would like to discuss what will happen to solutions of the
coupled RODEs system \eqref{3.1}-\eqref{3.2} as
$\lambda\rightarrow\infty$. First, we will give some lemmas which
play an important role in this section.

We need the following estimations. Suppose that
$(x_{\lambda}^{(1)}(t), x_{\lambda}^{(2)}(t), \cdots,
x_{\lambda}^{(N)}(t))^{\mathbf{T}}$  is a solution of the coupled
RODEs system \eqref{3.1}-\eqref{3.2}. For any two different
components $x_{\lambda}^{(j)}(t), x_{\lambda}^{(k)}(t)$ of the
solution for $\forall j, k\in \{1, 2, \ldots, N\}$,
\begin{eqnarray*}
d^{k, j}_{\lambda}(t)&=&2\langle x_{\lambda}^{(j)}(t)
-x_{\lambda}^{(k)}(t), F^{(j)}(x_{\lambda}^{(j)},
\bar{X}_t^{(j)})-F^{(k)}(x_{\lambda}^{(k)}, \bar{X}_t^{(k)})\rangle\\
&=& 2\langle x_{\lambda}^{(j)}(t)-x_{\lambda}^{(k)}(t),
f^{(j)}(x_{\lambda}^{(j)}+\bar{X}_t^{(j)})
-f^{(k)}(x_{\lambda}^{(k)}+\bar{X}_t^{(k)})\rangle\\
&&+2\|x_{\lambda}^{(j)}(t)-x_{\lambda}^{(k)}(t)\|^2 +2\langle
x_{\lambda}^{(j)}(t)-x_{\lambda}^{(k)}(t),
\bar{X}_t^{(j)}-\bar{X}_t^{(k)}\rangle\\
&\leq& -2l(\|x_{\lambda}^{(j)}(t)\|^2-\|x_{\lambda}^{(k)}(t)\|^2)
+2\|x_{\lambda}^{(j)}(t)-x_{\lambda}^{(k)}(t)\|^2\\
&&+2\langle  f^{(j)}(\bar{X}_t^{(j)})-f^{(k)}(\bar{X}_t^{(k)}),
 x_{\lambda}^{(j)}(t)-x_{\lambda}^{(k)}(t)\rangle\\
&&+2\langle x_{\lambda}^{(j)}(t)-x_{\lambda}^{(k)}(t),
\bar{X}_t^{(j)}-\bar{X}_t^{(k)}\rangle\\
&\leq & 2\|x_{\lambda}^{(j)}(t)-x_{\lambda}^{(k)}(t)\|
(\|f^{(j)}(\bar{X}_t^{(j)})\|+|\bar{X}_t^{(j)}|)\\
&&+2\|x_{\lambda}^{(j)}(t)-x_{\lambda}^{(k)}(t)
\|(\|f^{(k)}(\bar{X}_t^{(j)})\|+|\bar{X}_t^{(k)}|),
\end{eqnarray*}
thus, for fixed $\alpha>0$, we have
\begin{eqnarray*}
&&-\alpha \lambda \|x_{\lambda}^{(j)}(t)-x_{\lambda}^{(k)}(t)\|^2
+d^{k, j}_{\lambda}(t)\\
&\leq& \frac{1}{\lambda}
(\frac{4}{\alpha}\|f^{(j)}(\bar{X}_t^{(j)})\|^2)
+\frac{4}{\alpha}|\bar{X}_t^{(j)}|^2)+\frac{1}{\lambda}
(\frac{4}{\alpha}\|f^{(k)}(\bar{X}_t^{(k)})\|^2)+\frac{4}{\alpha}
|\bar{X}_t^{(k)}|^2).
\end{eqnarray*}
Let
$$C^{j, k, \alpha}_{T_1, T_2}(\lambda, \omega)=\frac{4}{\alpha}
\sup_{t\in [T_1,
T_2]}[(\|f^{(j)}(\bar{X}_t^{(j)})\|^2+|\bar{X}_t^{(j)}|^2)
+(\|f^{(k)}(\bar{X}_t^{(k)})\|^2+|\bar{X}_t^{(k)}|^2)]$$ in any
bounded interval $[T_1, T_2]$. Note that $\rho_{\lambda}(\omega)$ in
\eqref{3.9} satisfies
$$\frac{d}{d\lambda}\|\rho_{\lambda}(\omega)\|^2=2\langle
\rho_{\lambda}(\omega),
\frac{d}{d\lambda}\rho_{\lambda}(\omega)\rangle\leq 0,$$ and
consequently, $\rho_{\lambda}(\omega)\leq \rho_{1}(\omega)$ for
$\lambda\geq 1$. Hence, $C^{j, k, \alpha}_{T_1, T_2}(\lambda,
\omega)$ is uniformly bounded in $\lambda$ and
\begin{equation}-\alpha \lambda \|x_{\lambda}^{(j)}(t)
-x_{\lambda}^{(k)}(t)\|^2+d^{k, j}_{\lambda}(t)\leq
\frac{1}{\lambda}C^{j, k, \alpha}_{T_1, T_2}(\lambda, \omega)
\label{4.1}\end{equation} uniformly for $t\in [T_1, T_2]$ with
$$C^{j, k, \alpha}_{T_1, T_2}(\omega)=\sup_{\lambda\geq 1}
C^{j, k, \alpha}_{T_1, T_2}(\lambda, \omega).$$

Now let us estimate the difference between any two components of a
solution of the coupled RODEs system \eqref{3.1}-\eqref{3.2} as
$\lambda\rightarrow\infty$.

\begin{lemma} \label{sys2}
Provided conditions \eqref{1.2} and \eqref{1.3} are satisfied, then
any two components of a solution $(x_{\lambda}^{(1)}(t),
x_{\lambda}^{(2)}(t), \cdots, x_{\lambda}^{(N)}(t))^{\mathbf{T}}$ of
the coupled RODEs system \eqref{3.1}-\eqref{3.2} uniformly vanish in
any bounded time interval when the coupling coefficient
$\lambda\rightarrow\infty$, that is, for any bounded interval $[T_1,
T_2]$ and $\forall t\in [T_1, T_2]$, it yields
$$\lim_{\lambda\rightarrow\infty}\|x_{\lambda}^{(j)}(t)
-x_{\lambda}^{(k)}(t)\|=0, \ \ \forall j, k\in \{1, 2, \ldots,
N\}.$$
\end{lemma}

\begin{proof}
To prove the result, we can equivalently estimate the difference
between any two adjacent components only because the first and the
last components of the solution are considered to be adjacent. We
will notice that only one new term appears in each step which
continuous the process, except the last step that ends the process.

For the difference of the first part of the solution
$(x_{\lambda}^{(1)}(t), x_{\lambda}^{(2)}(t), \cdots,
x_{\lambda}^{(N)}(t))^{\mathbf{T}}$,
\begin{eqnarray}
\frac{d}{dt_+}\|x_{\lambda}^{(1)}(t)-x_{\lambda}^{(2)}(t)\|^2
&=&2\langle x_{\lambda}^{(1)}(t)-x_{\lambda}^{(2)}(t),
F^{(1)}(x^{(1)},
\bar{X}_t^{(1)})-F^{(2)}(x^{(2)}, \bar{X}_t^{(2)})\rangle\nonumber\\
&&-6\lambda \|x_{\lambda}^{(1)}(t)-x_{\lambda}^{(2)}(t)\|^2\nonumber\\
&&+2\lambda \langle x_{\lambda}^{(1)}(t)-x_{\lambda}^{(2)}(t),
x_{\lambda}^{(N)}(t)-x_{\lambda}^{(3)}(t)\rangle\nonumber\\
&\leq& -5\|x_{\lambda}^{(1)}(t)-x_{\lambda}^{(2)}(t)\|^2 +\lambda
\|x_{\lambda}^{(N)}(t)-x_{\lambda}^{(3)}(t)\|^2
+d^{1, 2}_{\lambda}(t)\nonumber\\
&\leq& -\beta\lambda\|x_{\lambda}^{(1)}(t)-x_{\lambda}^{(2)}(t)\|^2
+\lambda \|x_{\lambda}^{(N)}(t)-x_{\lambda}^{(3)}(t)\|^2\nonumber\\
&&+\frac{1}{\lambda}C^{1, 2, 5-\beta}_{T_1, T_2}(\omega)\label{4.2}
\end{eqnarray}
uniformly for $t\in [T_1, T_2]$ by \eqref{4.1}. Here, we can take
$$\beta=
\begin{array}{l}\begin{cases}
1-\cos\frac{N\pi}{N+2}, &\mbox{N is even},\\
1-\cos\frac{(N-1)\pi}{N+1}, &\mbox{N is odd}.
\end{cases}\end{array}
$$
In fact, from Lemma 4.1 in \cite{SZH}, we can take any $\beta\in
(-2\cos\frac{N\pi}{N+2}, 2)$ when $N$ is even and any $\beta\in
(-2\cos\frac{(N-1)\pi}{N+1}, 2)$ when $N$ is odd.

We have seen that the estimations in \eqref{4.2} generate
$x_{\lambda}^{(3)}(t)-x_{\lambda}^{(N)}(t)$. Now, we have
\begin{eqnarray*}
\frac{d}{dt_+} \|x_{\lambda}^{(3)}(t)-x_{\lambda}^{(N)}(t)\|^2&=&
2\langle x_{\lambda}^{(3)}(t)-x_{\lambda}^{(N)}(t), F^{(3)}(x^{(3)},
\bar{X}_t^{(3)})-F^{(N)}(x^{(N)}, \bar{X}_t^{(N)})\rangle\\
&&-4\lambda  \|x_{\lambda}^{(3)}(t)-x_{\lambda}^{(N)}(t)\|^2\\
&&+2\lambda \langle x_{\lambda}^{(3)}(t)-x_{\lambda}^{(N)}(t),
x_{\lambda}^{(2)}(t)-x_{\lambda}^{(1)}(t)\rangle\\
&&+2\lambda \langle  x_{\lambda}^{(3)}(t)-x_{\lambda}^{(N)}(t),
 x_{\lambda}^{(4)}(t)-x_{\lambda}^{(N-1)}(t)\rangle\\
&\leq&  -\beta\lambda\|x_{\lambda}^{(3)}(t)-x_{\lambda}^{(N)}(t)\|^2
+\lambda \|x_{\lambda}^{(1)}(t)-x_{\lambda}^{(2)}(t)\|^2\nonumber\\
&&+ \lambda\|x_{\lambda}^{(4)}(t)-x_{\lambda}^{(N-1)}(t)\|^2
+\frac{1}{\lambda}C^{3, N, 2-\beta}_{T_1, T_2}(\omega)
\end{eqnarray*}
uniformly for $t\in [T_1, T_2]$.

Note that $x_{\lambda}^{(1)}(t)-x_{\lambda}^{(2)}(t)$ has been fixed
and $x_{\lambda}^{(4)}(t)-x_{\lambda}^{(N-1)}(t)$ is generated.
Similarly, it yields
\begin{eqnarray*}
\frac{d}{dt_+}
\|x_{\lambda}^{(4)}(t)-x_{\lambda}^{(N-1)}(t)\|^2&\leq&
-\beta\lambda\|x_{\lambda}^{(4)}(t)-x_{\lambda}^{(N-1)}(t)\|^2+\lambda
\|x_{\lambda}^{(3)}(t)-x_{\lambda}^{(N)}(t)\|^2\nonumber\\
&&+ \lambda\|x_{\lambda}^{(5)}(t)-x_{\lambda}^{(N-2)}(t)\|^2
+\frac{1}{\lambda}C^{4, N-1, 2-\beta}_{T_1, T_2}(\omega)
\end{eqnarray*}
uniformly for $t\in [T_1, T_2]$.

Continue such estimations, for $j=2, 3, \ldots$, we get
\begin{eqnarray*} \frac{d}{dt_+} \|x_{\lambda}^{(j+3)}(t)
-x_{\lambda}^{(N-j)}(t)\|^2&\leq&
-\beta\lambda\|x_{\lambda}^{(j+3)}(t)
-x_{\lambda}^{(N-j)}(t)\|^2\\
&&+\lambda \|x_{\lambda}^{(j+2)}(t)
-x_{\lambda}^{(N-j+1)}(t)\|^2\nonumber\\
&&+ \lambda\|x_{\lambda}^{(j+4)}(t)-x_{\lambda}^{(N-j-1)}(t)\|^2
+\frac{1}{\lambda}C^{j+3, N-j, 2-\beta}_{T_1, T_2}(\omega)
\end{eqnarray*}
uniformly for $t\in [T_1, T_2]$.

We can divide the situation into two cases: $N$ is even and $N$ is
odd, which just as same as \cite{SZH} did. When $N$ is even, we can
rewrite the inequalities in the matrix form
\begin{equation}
\mathbf{\dot{u}}(t)\leq \mathbf{H}_{\lambda}\mathbf{u}(t)
+\frac{1}{\lambda}\mathbf{C}, \label{4.3}\end{equation} which
uniformly for $t\in [T_1, T_2]$, where for $t\in \mathbb{R}$,
$$\mathbf{u}(t)=(\|x_{\lambda}^{(1)}(t)-x_{\lambda}^{(2)}(t)\|^2,
\|x_{\lambda}^{(3)}(t)-x_{\lambda}^{(N)}(t)\|^2, \cdots,
\|x_{\lambda}^{(\frac{N}{2}+1)}(t)-x_{\lambda}^{(\frac{N}{2}+2)}(t)
\|^2)^{\mathbf{T}},$$
$$\mathbf{C}=(C^{1, 2, 5-\beta}_{T_1, T_2}(\omega), C^{3, N,
2-\beta}_{T_1, T_2}(\omega), \cdots, C^{\frac{N}{2}, \frac{N}{2}+3,
2-\beta}_{T_1, T_2}(\omega), C^{\frac{N}{2}+1, \frac{N}{2}+2,
5-\beta}_{T_1, T_2}(\omega))^{\mathbf{T}},$$ are
$\frac{N}{2}$-dimensional vectors, and
$$\mathbf{H}_{\lambda}= \left ( \begin{array}{cccccc}
     -\beta\lambda &\lambda  &0  &\cdots &0  \\
    \lambda & -\beta\lambda & \lambda  &\ddots & \vdots\\
    0&\lambda&\ddots &\ddots &0\\
    \vdots& \ddots& \ddots &-\beta\lambda &\lambda\\
    0& \cdots & 0 & \lambda &-\beta\lambda
\end{array} \right)_{\frac{N}{2}\times \frac{N}{2}}.$$
By Lemma \ref{Gronwall}, it follows from \eqref{4.3} that
\begin{equation}
\mathbf{u}(t)\leq e^{(t-t_0)\mathbf{H}_{\lambda}}\mathbf{u}(t_0)
+\frac{1}{\lambda}\int_{t_0}^te^{(t-s)\mathbf{H}_{\lambda}}\mathbf{C}ds.
\label{4.4}\end{equation} By Lemma 4.1 in \cite{SZH} again,
$\frac{1}{\lambda}\mathbf{H}_{\lambda}$ is negative definite, then
we have
$$\|e^{(t-t_0)\mathbf{H}_{\lambda}}\mathbf{u}(t_0)\|\leq
e^{(t-t_0)\mu_{\max}}\|\mathbf{u}(t_0)\|,$$ where
$\mu_{\max}=-\beta-2\cos\frac{N\pi}{N+2}<0$ is the maximal
eigenvalue of $\frac{1}{\lambda}\mathbf{H}_{\lambda}$. Thus
\eqref{4.4} implies that
$$\mathbf{u}(t)\rightarrow \mathbf{0} \ \ \mbox{as} \
\lambda\rightarrow\infty,$$ and
$$\|x_{\lambda}^{(1)}(t)-x_{\lambda}^{(2)}(t)\|^2\rightarrow 0\
\ \mbox{and}\ \ \|x_{\lambda}^{(\frac{N}{2}+1)}(t)
-x_{\lambda}^{(\frac{N}{2}+2)}(t)\|^2\rightarrow 0,$$ uniformly for
$t\in [T_1, T_2]$ as $\lambda\rightarrow\infty$.

Similarly, when $N$ is odd, we can rewrite the inequalities in the
matrix form
\begin{equation}
\mathbf{\dot{v}}(t)\leq \mathbf{\tilde{H}}_{\lambda}\mathbf{v}(t)
+\frac{1}{\lambda}\mathbf{\tilde{C}}, \label{4.5}\end{equation}
which uniformly for $t\in [T_1, T_2]$, where for $t\in \mathbb{R}$,
$$\mathbf{v}(t)=(\|x_{\lambda}^{(1)}(t)-x_{\lambda}^{(2)}(t)\|^2,
\|x_{\lambda}^{(3)}(t)-x_{\lambda}^{(N)}(t)\|^2, \cdots,
\|x_{\lambda}^{(\frac{N+1}{2})}(t)-x_{\lambda}^{(\frac{N+1}{2}
+2)}(t)\|^2)^{\mathbf{T}},$$
$$\mathbf{\tilde{C}}=(C^{1, 2, 5-\beta}_{T_1, T_2}(\omega),
C^{3, N, 2-\beta}_{T_1, T_2}(\omega), \cdots, C^{\frac{N-1}{2},
\frac{N+1}{2}+3, 2-\beta}_{T_1, T_2}(\omega), C^{\frac{N+1}{2},
\frac{N+1}{2}+2, 5-\beta}_{T_1, T_2}(\omega))^{\mathbf{T}},$$ are
$\frac{N-1}{2}$-dimensional vectors, and
$$\mathbf{\tilde{H}}_{\lambda}= \left ( \begin{array}{cccccc}
     -\beta\lambda &\lambda  &0  &\cdots &0  \\
    \lambda & -\beta\lambda & \lambda  &\ddots & \vdots\\
    0&\lambda&\ddots &\ddots &0\\
    \vdots& \ddots& \ddots &-\beta\lambda &\lambda\\
    0& \cdots & 0 & \lambda &-\beta\lambda
\end{array} \right)_{\frac{N-1}{2}\times \frac{N-1}{2}}.$$
By Lemma \ref{Gronwall}, it follows from \eqref{4.5} that
\begin{equation}
\mathbf{v}(t)\leq
e^{(t-t_0)\mathbf{\tilde{H}}_{\lambda}}\mathbf{v}(t_0)
+\frac{1}{\lambda}\int_{t_0}^te^{(t-s)\mathbf{\tilde{H}}_{\lambda}}
\mathbf{\tilde{C}}ds.\label{4.6}\end{equation} Just like the even
case, for uniform $t\in [T_1, T_2]$, we have
$$\|x_{\lambda}^{(1)}(t)-x_{\lambda}^{(2)}(t)\|^2\rightarrow 0, \
\ \mbox{as} \ \lambda\rightarrow\infty.$$ For other adjacent
components, the process above can be repeated. Hence, we can draw a
conclusion that the difference between any adjacent components of a
solution of the coupled RODEs system \eqref{3.1}-\eqref{3.2} tends
to zero uniformly for $t\in [T_1, T_2]$ as the coupling coefficient
goes to infinity which completes the proof.

\end{proof}

We know that all components of a solution of system
\eqref{3.1}-\eqref{3.2} have the same limit uniformly for $t\in
[T_1, T_2]$ as $\lambda\rightarrow\infty$. Now, we are in the
position to find what they converge to.

\begin{lemma}
If the assumptions \eqref{1.2} and \eqref{1.3} hold, then the random
dynamical system $\phi(t, \omega)$ generated by the solution of the
averaged RODE system
\begin{equation}
\frac{dZ}{dt_+}=\frac{1}{N}\sum_{j=1}^Nf^{(j)}(\bar{X}_t^{(j)}+Z)
+\frac{1}{N}\sum_{j=1}^N(\bar{X}_t^{(j)}+Z)\label{4.7}\end{equation}
has a singleton sets random attractor denoted by
$\{\bar{Z}(\omega)\}$. Furthermore,
$$\bar{Z}(\theta_t\omega)+\frac{1}{N}\sum_{j=1}^N\bar{X}_t^{(j)}$$
is the stationary stochastic solution of the equivalently averaged
SODE system
\begin{equation}
dz=\frac{1}{N}\sum_{j=1}^Nf^{(j)}(z)dt+\frac{1}{N}
\sum_{j=1}^Nc_jdL^{(j)}_t.\label{4.8}\end{equation}
\end{lemma}

\begin{proof}
Assume that $Z_1(t)$ and $Z_2(t)$ are two solutions of \eqref{4.7},
we have
$$\frac{d}{dt_+}\|Z_1(t)-Z_2(t)\|^2\leq (2-2l)\|Z_1(t)-Z_2(t)\|^2.$$
It follows from Gronwall's lemma that
$$\|Z_1(t)-Z_2(t)\|^2\leq e^{(2-2l)t}\|Z_1(0)-Z_2(0)\|^2,$$ which implies
$$\lim_{t\rightarrow\infty}\|Z_1(t)-Z_2(t)\|^2=0,$$
because of the Lipschitz coefficient $l>4$. Then all solutions of
\eqref{4.7} converge pathwise to each other.

Now, we have to give what they converge to based on the theory of
c\`adl\`ag random dynamical systems. Let $Z(t)$ be a solution of
\eqref{4.7}, we get
$$\frac{d}{dt_+}\|Z(t)\|^2\leq (4-2l)\|Z(t)\|^2+\frac{1}{N}
\sum_{j=1}^N\|f^{(j)}(\bar{X}_t^{(j)})\|^2+\frac{1}{N}
\sum_{j=1}^N|\bar{X}_t^{(j)}|^2.$$ From Gronwall's lemma, it yields
for $t>t_0$,
\begin{eqnarray*} \|Z(t)\|^2&\leq& e^{(4-2l)(t-t_0)}\|Z(t_0)\|^2\\
&&\ \ \ \ \ \ +\frac{1}{N}\sum_{j=1}^N\int_{t_0}^te^{(4-2l)(t-\tau)}
(\|f^{(j)}(\bar{X}_{\tau}^{(j)})\|^2+|\bar{X}_{\tau}^{(j)}|^2)d\tau.
\end{eqnarray*}
By pathwise pullback convergence with $t_0\rightarrow-\infty$, the
random closed ball centered as the origin with random radius
$\tilde{R}(\omega)$ is a pullback absorbing set of $\phi(t,
\omega)$, where
$$\tilde{R}^2(\omega)=1+\frac{1}{N}\sum_{j=1}^N\int_{-\infty}^0
e^{(2l-4)\tau}(\|f^{(j)}(\bar{X}_{\tau}^{(j)})\|^2
+|\bar{X}_{\tau}^{(j)}|^2)d\tau.$$ Obviously, by Lemma
\ref{property} and condition \eqref{1.3}, the integral defined in
the right-hand side is well-defined.

By Lemma \ref{attractor}, there exists a random attractor
$\{\bar{Z}(\omega)\}$ for $\phi(t, \omega)$. Since all solutions of
\eqref{4.7} converge pathwise to each other, the random attractor
$\{\bar{Z}(\omega)\}$ are composed of singleton sets.

Note that the averaged RODE \eqref{4.7} is transformed from the
averaged SODE \eqref{4.8} by the transformation
$$Z(t, \omega)=z-\frac{1}{N}\sum_{j=1}^N\bar{X}_{t}^{(j)},$$
so the pathwise singleton sets attractor
$\bar{Z}(\theta_t\omega)+\frac{1}{N}\sum_{j=1}^N\bar{X}_t^{(j)}$ is
a stationary solution of the averaged SODE \eqref{4.8} since the
Ornstein-Uhlenbeck process is stationary.
\end{proof}

Now, we will present another main result of this work.

\begin{theorem} (Synchronization under non-Gaussian L\'evy noise.) Let
$$(\bar{x}^{(1)}_{\lambda_n}(t, \omega), \bar{x}^{(2)}_{\lambda_n}
(t, \omega), \cdots, \bar{x}^{(N)}_{\lambda_n}(t,
\omega))^\mathbf{T} =(\bar{x}^{(1)}_{\lambda_n}(\theta_t\omega),
\bar{x}^{(2)}_{\lambda_n} (\theta_t\omega), \cdots,
\bar{x}^{(N)}_{\lambda_n}(\theta_t\omega)) ^\mathbf{T}$$ be the
singleton sets random attractor of the c\`adl\`ag random dynamical
system $\phi(t, \omega)$ generated by the solution of RODEs system
\eqref{3.1}-\eqref{3.2}, then
$$((\bar{x}^{(1)}_{\lambda_n}(t, \omega), \bar{x}^{(2)}_{\lambda_n}
(t, \omega), \cdots, \bar{x}^{(N)}_{\lambda_n}(t,
\omega))^\mathbf{T}) \rightarrow (\bar{Z}(t, \omega), \bar{Z}(t,
\omega), \cdots, \bar{Z}(t, \omega))^\mathbf{T}$$ in Skorohod metric
pathwise uniformly for $t$ belongs to any bounded time-interval
$[T_1, T_2]$ for any sequence $\lambda_n\rightarrow \infty$, where
$\bar{Z}(t, \omega)=\bar{Z}(\theta_t\omega)$ is the solution of the
averaged RODE \eqref{4.7} and $\bar{Z}(\omega)$ is the singleton
sets random attractor of the c\`adl\`ag random dynamical system
$\phi(t, \omega)$ which generated by the solution of averaged RODE
\eqref{4.7}.
\end{theorem}

\begin{proof} Define
\begin{equation}
\bar{Z}_{\lambda}(\omega)=\frac{1}{N}\sum_{j=1}^N
\bar{x}^{(j)}_{\lambda}(\omega),\label{4.9}\end{equation} where
$\{\bar{x}^{(1)}_{\lambda}(\omega), \bar{x}^{(2)}_{\lambda}(\omega),
\cdots, \bar{x}^{(N)}_{\lambda}(\omega)\}$ is the singleton sets
random attractor of the c\`adl\`ag RDS generated by RODEs system
\eqref{3.1}-\eqref{3.2}. Thus, $\bar{Z}_{\lambda}(t,
\omega)=\bar{Z}_{\lambda}(\theta_t\omega)$ satisfies
\begin{equation}
\frac{d\bar{Z}_{\lambda}(t, \omega)}{dt_+}=\frac{1}{N}
\sum_{j=1}^Nf^{(j)}(\bar{X}_t^{(j)}+\bar{x}^{(j)}_{\lambda}(t,
\omega))
+\frac{1}{N}\sum_{j=1}^N(\bar{X}_t^{(j)}+\bar{x}^{(j)}_{\lambda} (t,
\omega)),\label{4.10}\end{equation} Then, we get
\begin{eqnarray*}
\|\frac{d\bar{Z}_{\lambda}(t, \omega)}{dt_+}\|^2\leq \frac{2}{N}
\sum_{j=1}^N(\|f^{(j)}(\bar{X}_t^{(j)}+\bar{x}^{(j)}_{\lambda} (t,
\omega))\|^2+|\bar{X}_t^{(j)}+\bar{x}^{(j)}_{\lambda}(t,
\omega)|^2),
\end{eqnarray*}
by the c\`adl\`ag property of the solutions in \cite{Applebaum} and
the fact that these solutions belong to the compact ball
$\mathbb{B}_1(\omega)$, it follows that
$$\sup_{t\in [T_1, T_2]}\|\frac{d\bar{Z}_{\lambda}(t, \omega)}{dt_+}
\|\leq (\frac{2}{N}\sum_{j=1}^N\frac{\alpha}{4} \mathbf{C}_{T_1,
T_2}^{j, \bullet, \alpha}(\omega))^\frac{1}{2}<\infty.$$ By the
Ascoli-Arzel$\grave{a}$ theorem in $D([T_1, T_2], \mathbb{R}^d)$ in
\cite{Billingsley}, there exists a subsequence
$\lambda_{n_{k}}\rightarrow\infty$ such that
$\bar{Z}_{\lambda_{n_{k}}}(t, \omega)$ converges to $\bar{Z}(t,
\omega)$ in Skorohod metric as $n_k\rightarrow\infty$.

Since difference between any two components of a solution of the
coupled RODEs system \eqref{3.1}-\eqref{3.2} tends to zero uniformly
for $t\in [T_1, T_2]$ as $\lambda\rightarrow\infty$, from
\eqref{4.9}, we have
\begin{eqnarray*}
\bar{x}^{(j)}_{\lambda_{n_{k}}}(t, \omega)=\bar{Z}_{\lambda_{n_{k}}}
(t, \omega)+\frac{1}{N}\sum_{j'\neq j}\sum_{j''\neq j'}
(\bar{x}^{(j'')}_{\lambda_{n_{k}}}(t, \omega)-\bar{x}^{(j')}
_{\lambda_{n_{k}}}(t, \omega))\rightarrow \bar{Z}(t, \omega)
\end{eqnarray*}
uniformly for $t\in [T_1, T_2]$ as
$\lambda_{n_{k}}\rightarrow\infty$ for $j=1, \cdots, N$.
Furthermore, it follows from \eqref{4.10} that for $t\geq T_1$,
\begin{equation*}
\bar{Z}_{\lambda}(t, \omega)=\bar{Z}_{\lambda}(T_1, \omega)
+\frac{1}{N}\sum_{j=1}^N\int_{T_1}^t(f^{(j)}(\bar{X}_s^{(j)}
+\bar{x}^{(j)}_{\lambda}(s, \omega))+(\bar{X}_s^{(j)}
+\bar{x}^{(j)}_{\lambda}(s, \omega)))ds.
\end{equation*}
Thus,
\begin{eqnarray*}
 \bar{Z}(t, \omega)=\bar{Z}(T_1, \omega)+\frac{1}{N}
 \sum_{j=1}^N\int_{T_1}^t(f^{(j)}(\bar{X}_s^{(j)}+\bar{Z}(s, \omega))
 +(\bar{X}_s^{(j)}+\bar{Z}(s, \omega)))ds,
\end{eqnarray*}
uniformly for $t\in [T_1, T_2]$ as
$\lambda_{n_{k}}\rightarrow\infty$, which implies that
$\bar{Z}_{\lambda}(s, \omega)$ solves RODE \eqref{4.7}. Then, we
note that all possible sequences of $\bar{Z}_{\lambda_{n_{k}}}(t,
\omega)$ converges to the same limit $\bar{Z}(t, \omega)$ uniformly
for $t\in [T_1, T_2]$ as $\lambda_{n}\rightarrow\infty$. Since the
RDS generated by the solutions of RODE \eqref{4.7} has a singleton
sets random attractor $\{\bar{Z}(\omega)\}$, the stationary
stochastic process $\bar{Z}(\theta_t\omega)$ must be equal to
$\bar{Z}(t, \omega)$, i.e. $\bar{Z}(t,
\omega)=\bar{Z}(\theta_t\omega)$, which completes the proof.
\end{proof}

As a obvious result of Theorem \ref{sys1}, we get

\begin{corollary}
$$((\bar{x}^{(1)}_{\lambda}(t, \omega),
\bar{x}^{(2)}_{\lambda}(t, \omega), \cdots,
\bar{x}^{(N)}_{\lambda}(t, \omega))^\mathbf{T})\rightarrow
(\bar{Z}(t, \omega), \bar{Z}(t, \omega), \cdots, \bar{Z}(t,
\omega))^\mathbf{T}$$ in Skorohod metric pathwise uniformly for
$t\in [T_1, T_2]$ as $\lambda\rightarrow\infty$.
\end{corollary}

\begin{remark}
The results in this paper hold just in almost everywhere sense. In
the equation \eqref{1.1} we should replace the $X_t^{(j)}$ with
$X_{t_{-}}^{(j)}$ because we must take the left limit to make sure
that c\`adl\`ag solution process $X_t^{(j)}$ is predictable and
unique \cite{PZ}. For the typographical convenience, however, we
will use $X_t^{(j)}$ instead of  $X_{t_{-}}^{(j)}$ for the rest of
the paper. Moreover, in the case of additive noise, the distinction
for left limit or not is not necessary because if we have to
consider the integral form of equation \eqref{1.1},
$f^{(j)}(X_t^{(j)})$ has only countable discontinuous points and is
still Riemann and Legesgue integrable, where $j=1, \cdots, N$.
\end{remark}

\section{Conflict of Interests}
The authors declare that there is no conflict of interests regarding
the publication of this article.

%\section*{Acknowledgments}
%The author would like to thank the anonymous referees for their
%helpful comments and suggestions which largely improve the quality
%of the original manuscript.

%%For cite command type as \cite{1}; \cite{3,6} and \cite{2,4,6}.
%%For refcite command type as Refs.~[\cite{1}];
%%[\cite{1},\cite{3}] and [\cite{1}--\cite{4}].

\end{document}